\documentclass[11pt,draft,reqno]{amsart}
     \makeatletter
     \def\section{\@startsection{section}{1}%
     \z@{.7\linespacing\@plus\linespacing}{.5\linespacing}%
     {\bfseries
     \centering
     }}
     \def\@secnumfont{\bfseries}
     \makeatother
\newtheorem{theorem}{Theorem}[section]
\newtheorem{lemma}[theorem]{Lemma}
\newtheorem{proposition}[theorem]{Proposition}

\theoremstyle{definition}
\newtheorem{definition}[theorem]{Definition}
\usepackage{enumerate}
\usepackage{xcolor}

\usepackage{amssymb}
\newtheorem{remark}[theorem]{Remark}
\numberwithin{equation}{section}
\begin{document}

\title[LDP for a Class of Semilinear SPDEs in Any Space Dimension]{Large Deviations for a Class of Parabolic Semilinear Stochastic Partial Differential Equations in Any Space Dimension}


\author{Leila Setayeshgar}

\address{Leila Setayeshgar: Department of Mathematics $\&$ Statistics \\ Utah State University\\  3900 Old Main Hill, ANSC 202\\ 	Logan, UT 84322}
\email{{\tt leila.setayeshgar@usu.edu}} 
\address{Department of Mathematics and Computer Science, Providence College, Providence, RI 02918}
\email{\tt lsetayes@providence.edu}




\begin{abstract}
We prove the large deviation principle for the law of the solutions to a class of parabolic semilinear stochastic partial differential equations driven by multiplicative noise, in $C\big([0,T]:L^\rho(D)\big)$, where $D\subset {\mathbb R}^d$ with $d\geqslant 1$ is a bounded convex domain with smooth boundary and $\rho$ is any real, positive and large enough number.  The equation has nonlinearities of polynomial growth of any order, the space variable is of any dimension, and the proof is based on the weak convergence method.
\end{abstract}

\maketitle

\section{Introduction }

\noindent Let $D\subset {\mathbb R}^d$, with $d\geqslant 1$ be a bounded convex domain with smooth boundary $\partial D$.  We consider a family of nonlinear parabolic semilinear stochastic partial differential equations indexed by $0<\varepsilon\leqslant1$

\begin{eqnarray}\label{E:1}
\begin{array}{l}
{\displaystyle\frac{\partial}{\partial t} u^{\varepsilon}(t,x)=\bigg[\frac{\partial}{\partial x_i} \bigg(b_{ij}(x)\frac{\partial}{\partial x_j}u^{\varepsilon}(t,x) + g_i\big(t,x, u^{\varepsilon}(t,x)\big)\bigg)+f\big(t,x,u^{\varepsilon}(t,x)\big)\bigg]}\vspace{1 pt}\\
 \hspace{50pt} {\displaystyle+\sqrt{\varepsilon}{\sigma}_j\big(t, x, u^{\varepsilon}(t,x)\big)\frac{d}{dt}B^j, ~~~~~~~~~t\geqslant0, ~~x\in D,} 
 \vspace{7pt}\\ 
 \end{array}
 \end{eqnarray}
 
 \noindent with Dirichlet boundary conditions \\
 
 ${\displaystyle u^{\varepsilon}(t,x) = 0, ~~~t\geqslant 0, ~~x\in \partial D,}$\\

\noindent and initial condition\\

 ${\displaystyle u^{\varepsilon}(0,x) = \xi(x), ~~~x\in D.}$\\


\noindent Here $\displaystyle B^j := \{B^j(t), t\geqslant 0 , j=1,2, \cdots k\}$ is a $k$-dimensional Wiener process. The initial condition $\xi$, has a continuous stochastic modification and belongs to $L^{\rho}(D)$ for any $\rho$ real, positive and large enough.   The functions $f:=f(t, x, r)$, ${\sigma}_i := {\sigma}_i(t,x,r), ~i=1, 2, \cdots, k$ are locally Lipschitz in the third variable and have linear growth in $r\in {\mathbb R}$.  The function $g_i := g_i(t,x,r), ~i=1, 2, \cdots,d$ is locally Lipschitz in the third variable, and has polynomial growth of any order $\nu\geqslant 1$ in $r$.  Therefore, our family of semilinear equations contains, as special cases, both the stochastic Burgers' equation, and the stochastic reaction diffusion equation. The existence and uniqueness to equation (\ref{E:1}) has been proven by Gy\"ongy  and Rovira (2000) \cite{GR} via an approximation procedure. Our aim is to prove the large deviation principle for the law of the solutions to equation (\ref{E:1}) by employing the weak convergence method.  In the context of stochastic differential equations (SDEs), the Freidlin-Wentzell theory \cite{FW}, describes the {{asymptotic behavior}} of probabilities of the large deviations of the law of the solutions to a family of small noise finite dimensional SDEs, away from its law of large number limit. In this work we deal with the case where the noise term is infinite dimensional.  In \cite{BDM},  Budhiraja et al. (2008) use certain {{variational representations}} for infinite dimensional Brownian motions \cite{BD} based on the work of Bou\'e and Dupuis \cite{BoD} and show that, these representations provide a framework for proving large deviations for a variety of infinite dimensional systems, such as stochastic partial differential equations (SPDEs).  One of the advantages of their method is that, the technical exponential probability estimates needed to justify certain approximations are bypassed.  Instead, one is required to prove certain qualitative properties of the SPDE under study.  The following is the main contribution of this paper and establishes the large deviation principle for the law of the solutions to equation (\ref{E:1}).

\begin{theorem}[Main Theorem]\label{LDP}

\noindent The processes $\{u^{\varepsilon}(x, t): x\in D, ~~t\in[0,T]\}$ satisfy the large deviation principle on $C\big([0,T]: L^{\rho}(D)\big)$ with rate function $I_{\xi}$ given by (\ref{ratefunction}).
\end{theorem}

\noindent For the conception of the proof we refer the reader to \cite{BD}.  Note that for carrying out the proof Theorem 6 in Budhiraja, Dupuis and Maroulas \cite{BDM} is utilized.  We defer the precise definition of the rate function to Section 3.  The main difference between the current work and that of \cite{GR} is that here the driving noise is finite dimensional as opposed to infinite dimensional and the space variable is of several dimensions as opposed to one dimensional.

\subsection{Outline of the paper.}
 In Section 2 we state some assumptions and preliminaries.  The existence and uniqueness results for the family of semilinear SPDEs is also stated in this section.  In Section 3 we state the large deviations theorem due to Budhiraja, Dupuis and Maroulas (\cite[Theorem 7]{BDM}) which we exploit.  In Section 4 we introduce the controlled and the skeleton equations and establish their existence and uniqueness.  Section 5 is devoted to the proof of the main theorem.  Establishing the large deviations principle hinges on proving the tightness and convergence of the controlled process. This is carried out in Proposition 5.1.
 Finally, proof of the main theorem follows readily by verification of assumptions (S1) and (S2).
\subsection{Notation.}
{ Unless otherwise noted, we adopt the following notation throughout the paper.  The summation convention is in place.  The notation $:=$ means by definition.    $C$ denotes a  {{free}} constant which may take on different values, and depend upon other parameters. If $\theta$ is a vector, then $\theta_i$ denotes the $i^{th}$ component of that vector. We use the notation $|h(t,\cdot)|_p =|h(t)|_p$ to denote the $L^p({\mathbb R}^d)$-norm of a function $h = h(t,x)$ with respect to the variable $x\in {\mathbb R}^d$.  If $h(t,x)$ is only defined for $x\in D$, then $|h(t)|_p$ denotes the $L^p(D)$ norm.  If $h=h(t,x)$ is a random field and $X$ assumes a value in a functional space, then saying that almost surely $h$ is in $X$ means that $h$ has a stochastic modification which is in $X$, almost surely.}

\noindent\section{Preliminaries}

\noindent In this section we introduce a set of assumptions and preliminaries that are necessary for the formulation of the problem.  Let  $(\Omega, {\mathcal F}, {\mathbb F}, P)$ be a filtered probability space or stochastic basis carrying a $k$-dimensional Brownian motion $\{B^j(t), t\geqslant 0 , 1\leqslant j\leqslant k\}$, with the filteration ${\mathbb F}=\{{\mathcal F}_{t}\}_{t\in[0,T]}$.  The following are some main assumptions that are in effect throughout the paper: \\

\begin{enumerate}[(A1)]

\item [(A1)] The domain $D\subset {\mathbb R}^d, ~d\geqslant 1$ is a bounded convex set with smooth boundary.\\

\item  [(A2)] The matrix $b_{ij}(x) \in C^2(\bar D)$ is symmetric for every $x\in D$, and satisfies the uniform ellipticity condition, i.e.,

$$\frac{1}{\kappa}|\mu|^2\geqslant b_{ij}(x){\mu}_i{\mu}_j\geqslant \kappa|{\mu}|^2, ~~~~~\forall ~~\mu\in {\mathbb R^d}, ~~~~x\in D,$$

\noindent with a constant $\kappa>0$. \\

\item [(A3)] The functions $g_i$ are of the form $g_i(t,x,r) := g_{i1}(x,t,r)+g_{i2}(t,r)$, where $g_{i1}$ and $g_{i2}$ are Borel functions of $(t,x,r)\in {\mathbb R}_+\times D \times {\mathbb R}$ and of $(t,r)\in {\mathbb R}_+\times {\mathbb R}$, respectively.  Moreover, for every $T\geqslant 0$ there is a constant $K$ such that

$$|g_{i1}(t,x,r)|\leqslant K(1+|r|),~~~~~~~~~~~~|g_{i2}(t,r)|\leqslant K(1+|r|^{\nu}),$$

\noindent for all $t\in [0,T]$, $x\in D$, $r\in {\mathbb R}$, with some $\nu\geqslant 1$, and for all $~~i = 1, \cdots, d$.\\

\item [(A4)] The functions $f:=f(t,x,r)$, ${\sigma}_j :={\sigma}_j(t,x,r),~~~~~~~~ j=1, \cdots k$ are Borel functions and have linear growth in $r$, i.e., \\

\noindent For every $T\geqslant 0$ there exists a constant $L$ such that

\begin{align*}
&\sum_j|\sigma_j(t,x,r)|^2\leqslant L(|r|^2+1), ~~~~~j=1,\cdots,k,\\&|f(t,x,r)|\leqslant L(|r|+1), 
\end{align*}

\noindent for all $t\in [0,T]$, $x\in D$, and $r,s \in {\mathbb R}$.

\item [(A5)] For every $T\geqslant 0$, there exists a constant $L$ such that

\begin{align*}
& \sum_j|\sigma_j(t,x,r)-\sigma_j(t,x,s)|^2\leqslant L(|r-s|^2), ~~~j=1,\cdots,k,\\& |f(t,x,r)-f(t,x,s)|\leqslant L|r-s|,\\&  |g_i(t,x,r)-g_i(t,x,s)|\leqslant L(1+|r|^{\nu-1}+|s|^{\nu-1})|r-s|, ~~~i=1,\cdots,d,
\end{align*}

\noindent for all $t\in [0,T]$, $x\in D$, $r,s \in {\mathbb R}$.

\end{enumerate}
\begin{definition}[Mild solution]\label{definition} \rm{A random field $u^{\varepsilon} :=  \{u^{\varepsilon}(t,x) : t\in[0,T], x\in D\}$ is called a mild solution of equation (\ref{E:1}) with initial condition $\xi\in L^\rho(D)$ if $u^{\varepsilon}(\cdot,\cdot)$ is $C([0,T]: L^\rho(D))$-valued a.s. and $u^{\varepsilon}(t,x)$ is $\{{\mathcal F}_t\}$-measurable for any $t\in[0,T]$, and $x\in D$, and if

\begin{align*}
\label{E:Semilinear_Solution_definition}
u^{\varepsilon}(t,x)&= \int_D G_t(x,y)\xi(y)dy + \sqrt{\varepsilon} \int_0^t\int_D G_{t-s}(x,y) \sigma_j(s, u^{\varepsilon}(s))(y)dydB^j(s)\\& -\int_0^t\int_0^1\partial_{y_i}G_{t-s}(x,y)g_i(s, u^{\varepsilon}(s))(y)dyds\\&+\int_0^t\int_D G_{t-s}(x,y)f(s, u^{\varepsilon}(s))(y)dyds.
\end{align*}}
\end{definition}

\noindent The function $G_t(x,y)$, $t\geqslant 0,~~ x,y \in D$ is the Green kernel associated with the following linear equation\\

$\displaystyle \frac{\partial}{\partial t} u(t,x)= \frac{\partial}{\partial x_i} \bigg(b_{ij}(x)\frac{\partial}{\partial x_j}u(t,x)\bigg),$\\ 

\noindent with Dirichlet's boundary condition\\

$\displaystyle u(t,x) =0, ~~t\geqslant 0, ~x\in \partial D,$\\ 

\noindent where $b_{ij}\in C^2(\bar D)$, and $\partial D$ is Lipschitz.   We now state some estimates on the Dirichlet heat kernel ({\cite[Proposition 3.5]{GR}}, \cite{L}).  

\subsection{Estimates on the Dirichlet heat kernel}
 
\noindent There exist Borel functions $a, b, d$ and some constants $K, C>0$ such that for some $p\geqslant 1$ and for all $0\leqslant s<t\leqslant T,$ $x, y\in D$
 
 \begin{enumerate}[(E1)]
 
 \item [(E1)]$\displaystyle |D^\delta_xG_{t-s}(x,y)|\leqslant a(t-s, x-y), ~~~~~~~~~~~~~|a(t, \cdot)|_p\leqslant K_pt^{-1+{\kappa}_p}$,\\
 
 \item [(E2)]$\displaystyle |\frac{\partial}{\partial x_i} D_x^\delta G_{t-s}(x,y)|\leqslant b(t-s, x-y), ~~~~~~~~~~~~~|b(t, \cdot)|_p\leqslant K_pt^{-1-{\theta}_p+{\kappa}_p}$,\\
 
 \item [(E3)]$\displaystyle |\frac{\partial}{\partial s}D_x^\delta G_{t-s}(x,y)|\leqslant c(t-s, x-y), ~~~~~~~~~~~~~|c(t, \cdot)|_p\leqslant K_pt^{-1-{v}_p+{\kappa}_p}, $
 
 \noindent where $\displaystyle {\kappa}_p := \frac{1}{2}(d/p-d+2-|\delta|), ~~\theta_p:=(|\delta|+1)/2, ~~~v_p:=(|\delta|+2)/2.$\\

\item [(E4)] $\displaystyle |D_t^nD_x^\delta G_x(t; x,y)|\leqslant K t^{-(d+2n+|\delta|)/2} \exp\bigg(-C\frac{|x-y|^2}{t-s}\bigg)$

\noindent for $2n+|\delta|\leqslant 3$, where $D_t^n := \partial_n/\partial t^n, ~D_x^\delta := \partial^\delta_1/\partial {x_1}^{\delta_1}\cdots \partial^\delta_d/\partial {x_d}^{\delta_d}, \\~\delta := (\delta_1, \cdots \delta_d)$ is a multi-index, $|\delta| := \delta_1+\delta_2+\cdots\delta_d$.  
\end{enumerate}

\begin{remark}
{\it Due to $A := (\partial/\partial x_i)(b_{ij}(x)\partial/\partial x_j) := A^*$, $G_{t-s}(x,y)$ is symmetric in $x,y$. Therefore (E4) also holds with $D_y$ in place of $D_x$.}
 \end{remark}

\noindent The following Theorem (\cite[Theorem 2.1]{GR}) asserts the existence of a unique solution to equation (\ref{E:1}).

\begin{theorem}[Existence $\&$ uniqueness of solution mapping]\label{existence}
Assume the set of Hypotheses (A).  Then there exists ${\rho}_0 :={\rho}_0(\nu,d)$, such that for every $\rho>{\rho}_0$ equation (\ref{E:1}) has a unique $C([0,T]: L^{\rho}(D))$-valued solution, provided $\xi$ is an ${\mathcal {F}}_0$-measurable, $L^{\rho}(D)$-valued random element.  Moreover, if $\xi$ has a continuous stochastic modification, then $u^\varepsilon(t,x)$ has a stochastic modification which is continuous in $(t,x) \in[0,\infty)\times D$.
\end{theorem}

 \begin{remark}
{\it The proof of the above theorem is based on an approximation procedure where the nonlinear functions $f$, $g_i$ in equation (\ref{E:1}) are approximated by bounded Lipschitz functions.  It is proven that the solutions to the approximations of equation (\ref{E:1}) are bounded in probability in  $C([0,T] \times D)$ where an energy equality is used.  This results in the tightness of the sequence of approximated solutions in $C([0,T]: L^{\rho}(D))$.  It is then shown that the sequence of approximated solutions converges in probability to the unique solution of equation (\ref{E:1}).} 
 \end{remark}


\noindent The following three lemmas are used in proving the main theorem.

\begin{lemma} [{\cite[Corollary 3.6]{GR}}]\label{Gyongy-s} 
Set
$$\Upsilon(\phi^{\varepsilon})(t,x) := \int_0^t\int_D G_{t-s}(x,y) \phi^{\varepsilon}_j(t,y) dy dB^j(s), ~~~t\in[0,T], ~~x\in D,$$
where $\displaystyle \phi^\varepsilon := \{\phi^{\varepsilon}(t,x) = (\phi^{\varepsilon}_1(t,x) , \cdots ,\phi^{\varepsilon}_k(t,x) ): t\in[0,T], x\in D\}$ is a sequence of $\mathcal F_t$-adapted random fields.  If for $\rho>d$ we have a constant $C$ such that $\displaystyle |\phi^{\varepsilon}|_\rho \leqslant C$ for all $t\in [0,T]$ and for all $\varepsilon \in (0,1]$, then $\Upsilon(\phi^{\varepsilon})$ is tight in $C([0,T]\times D)$, uniformly in $\varepsilon$. In general, there is a number $\bar \rho>d$ such that if for $\rho>\bar \rho$ we have $\displaystyle \sup_{\varepsilon>0}{\mathbb E} (\sup_{t\leqslant T} |\phi^{\varepsilon}(t)|_\rho^\rho)<\infty$.  Then $\Upsilon(\phi^\varepsilon)$ is tight in $C([0,T]\times D),$ uniformly in $\varepsilon$.
\end{lemma}

 
 \noindent Let $q\geqslant 1$ and  $\displaystyle R(r,t;x,y) := \partial_yG_{r-t}(x,y)$ or $\displaystyle G(r,t;x,y)$ for $t\in [0,T]$ and $x\in D$.  For $v\in L^{\infty}([0,T]: L^q(D))$ define the linear operator $J$ by\\

$\displaystyle J(v)(t,x) := \int_0^t\int_D R(r,t;x,y,) v(r,y) dy dr, ~~~~~~~~~~t\in[0,T], ~~~x\in D,$\\

$\displaystyle J(v)(t,x) := 0 ~~~~{\mbox {if}} ~~x\notin D,$\\

\noindent provided the integral exists.

\begin{lemma}[{\cite[Corollary 3.2]{GR}}] \label{Gyongy}
Assume (E1)-(E3) with $\kappa_p>0$.  Let $\zeta_n(t,x)$ be a sequence of random fields on $[0,T]\times D$ such that almost surely

$$|\zeta_n(t,\cdot)|_q\leqslant \vartheta_n, ~~~~{\mbox{for all}}~~ t\in [0,T],$$

\noindent where $\vartheta_n$ is a finite random variable for every $n$.  Assume that the sequence $\vartheta_n$ is bounded in probability, i.e.

$$\lim_{C\to \infty} \sup_{n\in {\mathbb N}}P(\vartheta_n\geqslant C) =0$$

\noindent Then for $0\leqslant \alpha<\min(\kappa_p/v_p, \kappa_p)$, the sequence $J(\zeta_n)$ is uniformly tight in $C^\alpha([0,T]: L^\rho(D))$.  In the case $\rho=\infty$ the sequence $J(\zeta_n)$ is tight in the  $C^{\alpha, \beta}([0,T]: L^\rho(D))$ for $0\leqslant \alpha< \min (\kappa_p/v_p,\kappa_p)$, $0\leqslant \beta<\min(\kappa_p/v_p, 1)$.
\end{lemma}

\begin{lemma}[{\cite[Lemma 3.1] {GR}}] \label{lem:GR-3.1}
 Let $\rho \in [1,\infty]$, $q \in  [1,\rho]$, and $p:=(1+1/p-1/q)^{-1}$. Let $\displaystyle\kappa_p:= \frac{1}{2}(d/p-d+2-|\delta|)$. Let $\alpha_0 := \min\{\kappa_p/v_p, \kappa_p\}$, $\beta_0:= \min\{\kappa_p/\theta_p, 1/\rho\}$. The following statements hold.
 \begin{enumerate}[(i)]
   \item $J$ is a bounded linear operator from $L^\gamma([0,T]:L^q(D)) \to C([0,T]:L^\rho(D))$ for every $\gamma>1/\kappa_p$ and there exists $C>0$ such that
       \begin{equation} \label{eq:GR-3.1-i}
         |J(v)(t,\cdot)|_\rho \leq C\int_0^t(t-s)^{\kappa_p-1}|v(s)|_qds \leq C t^{\kappa_\rho - \frac{1}{\gamma}} \left(\int_0^t |v(s)|_q^\gamma ds \right).
       \end{equation}
   \item For every $\alpha \in (0,\alpha_0)$, $\gamma> (\alpha_0-\alpha)^{-1}$ there exists $C>0$ such that
       \begin{equation} \label{eq:GR-3.1-ii}
         |J(v)(t,\cdot) - J(v)(s,\cdot)|_\rho \leq C|t-s|^\alpha \left(\int_0^{t\vee s} |v(r)|_q^\gamma \right)^{\frac{1}{\gamma}}.
       \end{equation}
   \item Let $\beta \in (0,\beta_0)$. Assume that $\kappa_{p'}>0$ for $p':=(1/p - \beta)^{-1}$. Then for every $\gamma>(\beta_0 - \beta)^{-1}$, there exists $C>0$ such that
       \begin{equation} \label{eq:GR-3.1-iii}
         |J(v)(t,\cdot) - J(v)(t,\cdot + \zeta)|_\rho \leq C|\zeta|^\beta \left(\int_0^t |v(r)|_q^\gamma dr \right)^{\frac{1}{\gamma}}.
       \end{equation}
   \item If $\rho=\infty$ then for $0< \beta < \min\{\kappa_p/\theta_p,1\}=:\beta'$ and $\gamma>(\beta'-\beta)^{-1}$, there is a constant $C>0$ such that
       \begin{equation} \label{eq:GR-3.1-iv}
         |J(v)(t,\xi) - J(v)(t,\xi+\zeta)| \leq C|\zeta|^\beta \left(\int_0^t |v(r)|^\gamma_q \right)^{\frac{1}{q}}.
       \end{equation}
 \end{enumerate}
\end{lemma}
\section{The large deviations principle }

In this section, we state Theorem \ref{LDP_original} {\cite[Theorem 6]{BDM}}, which asserts the uniform Laplace principle for a family of functionals of a cylindrical Wiener process under two main assumptions.  Let $(\Omega, {\mathcal F}, {\mathbb F}, P)$ be the filtered probability space introduced as before. 
Denote by ${\mathcal A}_2$ the set of predictable processes which belong to $L^2(\Omega \times [0,T] :{\mathbb R}^k)$.  For any $N>0, N \in {\mathbb N}$, define

\begin{align*}
\Lambda^N & :=  \left\{\tau \in L^2([0,T]: {\mathbb R}^k): \int_0^T |\tau|^2ds  \leq N\right\}, \\
{\mathcal{A}}^N_2 & := \left\{\eta(\omega) \in \Lambda^N: \eta\in   {\mathcal A}_2 ,  ~~ P-a.s. \right\},
 \end{align*}
 
\noindent where $|\cdot|$ is the norm in ${\mathbb R}^k$.  Note that $\Lambda^N$ is a compact metric space endowed with the weak topology from $L^2([0,T]: {\mathbb R}^k)$. The space ${\mathcal A}_2^N$  is the space of admissible controls, and plays an important role in the weak convergence approach to the theory of large deviations.  Let ${\mathcal E}_0$ and ${\mathcal E}$ be Polish spaces.  Assume that the initial condition takes values in a compact subspace of ${\mathcal E}_0$, and denote the solution space by ${\mathcal E}$.  For every $\varepsilon\in(0,1]$, let  ${\mathcal H}^{\varepsilon} : {\mathcal E}_0 \times  {\mathcal C}([0,T]: {\mathbb R}^{k}) \rightarrow {\mathcal E}$ be a family of measurable maps, and define $Y^{\varepsilon}_\xi :=  {\mathcal H}^{\varepsilon} (\xi, \sqrt{\varepsilon }B)$.  For a control $\varphi\in {\mathcal{A}}_2^N$, and under the measure $Q$ defined by
$$
\frac{dQ}{dP} := \exp\left\{-\frac{1}{\sqrt{\varepsilon}}\int_0^T \langle \varphi(s) , dB(s) \rangle -\frac{1}{2\varepsilon}\int_0^T|\varphi(s)|^2ds\right\},
$$
\noindent Girsanov's theorem implies that the process
\begin{align*}
  \bar{B}(t) := B(t) + {\varepsilon}^{-1/2} \int_0^t \varphi(s) ds,
\end{align*}
is a $k$-dimensional Wiener process.  The following is the standing assumption of Theorem \ref{LDP_original}, the large deviations principle of \cite{BDM}. \\

\noindent ASSUMPTION:  There exists a measurable map ${\mathcal H}^0: {\mathcal E}_0 \times L^2([0,T]: {\mathbb R}^k) \rightarrow {\mathcal E}$, such that\\
 
 \begin{enumerate}[(S1)]
   \item For every $M<\infty$ and compact set $K \subset {\mathcal E}_0$, the set 
 $$
 \Gamma_{M,K} := \bigg\{{\mathcal H}^0\bigg(\xi, \int_0^t \varphi(s) ds\bigg): \varphi \in \Lambda^M, \xi\in K\bigg\},
 $$
 is a compact subset of ${\mathcal E}$.\\
 
  \item Consider $M<\infty$ and the families $\{\varphi^{\varepsilon}\}_{\varepsilon>0} \subset {\mathcal A}^M_2$ and $\{\xi^{\varepsilon}\}_{\varepsilon>0} \subset {\mathcal E}_0$, such that $\varphi^{\varepsilon}\rightarrow \varphi$, and $\xi^{\varepsilon} \rightarrow \xi$ in distribution, as $\varepsilon \rightarrow 0$. Then
$$
{\mathcal H}^{\varepsilon}\bigg(\xi^{\epsilon}, \sqrt{\varepsilon} B + \int_0^t \varphi^{\varepsilon}(s) ds\bigg) 
\longrightarrow 
{\mathcal H}^0\bigg(\xi, \int_0^t \varphi (s) ds \bigg), 
$$
\noindent in distribution as $\varepsilon \rightarrow 0$.\\
 \end{enumerate}

For $\psi\in{ \mathcal E}$, and $\xi \in {\mathcal E}_0$, define the rate function
\begin{equation}\label{E:rate function}
I_\xi(\psi) := \inf_{{\{\beta\in L^2([0,T]: {\mathbb R}^k) :\, \psi := {\mathcal H}^0(\xi, \int_0^t \beta(s) ds)\}}} \left\{\frac{1}{2}\int_0^T |\beta(s)|^2 ds\right\},
\end{equation}

\noindent where $\inf(\varnothing)=+\infty$.\\

\noindent The following theorem states the uniform Laplace principle for the family $\{Y^{\varepsilon}_ \xi\}$.

\begin{theorem} [{\cite[Theorem 6]{BDM}}] \label{LDP_original}
Let ${\mathcal H}^0: {\mathcal E}_0 \times C([0,T]: {\mathbb R}^d) \rightarrow {\mathcal E}$ be a measurable map satisfying assumptions (S1) and (S2).  Suppose that for all $f\in{\mathcal E}$, $\xi\mapsto I_{\xi}(f)$ is a lower semi-continuous map from ${\mathcal E}_0$ to $[0, \infty]$. Then for every $\xi \in {\mathcal E}_0 $, $I_{\xi}(\psi): {\mathcal E}\rightarrow [0,\infty]$, is a rate function on ${\mathcal E}$ and the family $\{I_{\xi}, \xi \in \mathcal{E}_0\}$ of rate functions has compact level sets on compacts.  Furthermore, the family $\{Y^{\varepsilon}_\xi\}$ satisfies the uniform Laplace principle on ${\mathcal E}$  with rate function $I_{\xi}$, uniformly in $\xi$ on compact subsets of ${\mathcal E}_0$.

\end{theorem}

\begin{remark}
{\it  In order to prove Theorem \ref{LDP}, it suffices to verify assumptions (S1) and (S2) with ${\mathcal E}_0 := L^\rho(D)$  and ${\mathcal E} := C([0,T]: L^\rho(D))$.}
\end{remark}


\section{Controlled and skeleton equations}  
\noindent The solution to equation (\ref{E:1}) is the random field $u^\varepsilon$ which we set to be the map ${\mathcal H}^\varepsilon(\xi, \sqrt\varepsilon B)$.   $v^{\varepsilon,\varphi}_{\xi}(t, x) $ which we denote by the map ${\mathcal H}^\varepsilon\big(\xi, \sqrt\varepsilon B+ \int_0^t\varphi(s)ds\big)$ is referred to as the controlled equation (or process) with the following mild form
\begin{align}
\label{E:controlled_process}
v^{{\varepsilon}, \varphi}_{\xi}(t,x)=& \nonumber \int_D G_t(x,y)\xi(y)dy  -\int_0^t\int_0^1\partial_{y_i}G_{t-s}(x,y)g_i(s, v^{{\varepsilon}, \varphi}_{\xi}(s))(y)dyds \\&\nonumber+\int_0^t\int_D G_{t-s}(x,y)f(s, v^{{\varepsilon}, \varphi}_{\xi}(s))(y)dyds\\& + \varepsilon \int_0^t\int_D G_{t-s}(x,y) \sigma_j(s, v^{{\varepsilon}, \varphi}_{\xi}(s))(y)\varphi_j(s)dyds\nonumber\\& + \sqrt{\varepsilon} \int_0^t\int_D G_{t-s}(x,y) \sigma_j(s, v^{{\varepsilon}, \varphi}_{\xi}(s))(y)dydB^j(s) 
\end{align}

\noindent $v^{0,\varphi}_{\xi}(t, x)$ denoted by the map  ${\mathcal H}^0\big(\xi, \int_0^t \varphi(s) ds)\big)$ is referred to as the skeleton equation with the following mild form
  
\begin{align}
\label{E:limiting_process}
\nonumber v^{0, \varphi}_{\xi}(t,x)=& \int_D G_t(x,y)\xi(y)dy -\int_0^t\int_D\partial_{y_i}G_{t-s}(x,y)g_i(s, v^{0, \varphi}_{\xi}(s))(y)dyds\\&\nonumber+\int_0^t\int_D G_{t-s}(x,y)f(s, v^{0, \varphi}_{\xi}(s))(y)dyds\\&+ \int_0^t\int_D G_{t-s}(x,y) \sigma_j(s, v^{0, \varphi}_{\xi}(s))(y)\varphi_j(s)dyds.
\end{align}

\subsection{The rate function} Let $\psi\in C\big([0,T]: L^\rho(D)\big)$, for every $t\in [0, T]$, and $x\in D$. Define the following rate function (or action functional)

\begin{equation}\label{ratefunction}
I_\xi(\psi) := \frac{1}{2}\inf_{\beta}\int_0^T|\beta(s)|^2 ds,
\end{equation}

\noindent where the infimum is taken over all $\beta\in L^2([0,T] : {\mathbb R}^k)$ such that

\begin{align}
\label{E:Deterministic}
&\nonumber \psi(t,x)= \int_D G_t(x,y)\xi(y)dy -\int_0^t\int_D\partial_{y_i}G_{t-s}(x,y)g_i(s, \psi(s))(y)dyds\\&\nonumber+\int_0^t\int_D G_{t-s}(x,y)f(s, \psi(s))(y)dyds\\&+ \int_0^t\int_D G_{t-s}(x,y) \sigma_j(s, \psi(s))(y)\beta_j(s)dyds.
\end{align}

\subsection{Existence and uniqueness of controlled process.} 

\noindent The following theorem  (\cite[Theorem 10]{BDM}), asserts the existence and uniqueness of the controlled process, with the main ingredient of proof being the Girsonov's theorem \cite[Theorem 10.14]{DZ}.

\begin{theorem}[Existence $\&$ uniqueness of controlled process]\label{existence_controlled} Let ${\mathcal H}^{\varepsilon}$ denote the solution mapping, and let $\varphi\in {\mathcal A}^N_2$ for some $N\in {\mathbb N}$.  For $\varepsilon>0$ and $\xi \in L^\rho(D)$ define
$$
v^{\varepsilon, \varphi}_{\xi} := {\mathcal H}^{\varepsilon}\bigg(\xi, \sqrt{\varepsilon}B+ \int_0^t\varphi(s) ds\bigg),
$$
then $v^{\varepsilon, \varphi}_{\xi}$ is the unique solution of equation (\ref{E:controlled_process}).
\end{theorem}
\subsection{Existence and uniqueness of skeleton.} 
\noindent The next Theorem shows the existence and uniqueness of the skeleton equation whose proof is almost verbatim to that of Theorem \ref{existence}, and thus omitted.
\begin{theorem}[Existence $\&$ uniqueness of skeleton]\label{uniqueness} Fix $\xi\in L^{\rho}(D)$ and $\varphi \in L^2([0,T] : {\mathbb R}^k)$.  Then there exists a unique function $\psi\in \mathcal C\big([0,T]: L^{\rho}(D)\big)$ which satisfies equation (\ref{E:Deterministic}).
\end{theorem}
\vspace{-5 pt}
\section{Proof of Theorem \ref{LDP}}
\noindent The following proposition plays a key role in proving Theorem \ref{LDP}.  It leads to the verification of assumptions (S1) and (S2).  To this purpose, let $\gamma:[0,1)\to[0,1)$ be a measurable map such that $\gamma(r)\to\gamma(0)=0$ as $r\to0$.

 \begin{proposition}[Convergence of controlled process] \label{convergence}
 Let $M<\infty$, and suppose that $\xi^{\varepsilon}\rightarrow\xi$ and $\varphi^{\varepsilon}\rightarrow \varphi$ in distribution as $\varepsilon\rightarrow 0$ with $\{\varphi^{\varepsilon}\}_{\varepsilon>0} \subset {\mathcal A}_2^M$.  Then $v^{{\gamma
 (\varepsilon)}, \varphi^{\varepsilon}}_{\xi^{\varepsilon}} \rightarrow v^{0,\varphi}_{\xi}$ in distribution.
 \end{proposition}
 
\begin{proof} We carry out the proof in two steps.\\

\noindent {\bf Step 1: Tightness }\\

\noindent In this section, we show that  $v^{\gamma(\varepsilon), \varphi^\varepsilon}_{\xi^\varepsilon}$ is tight in $C\big([0,T]: L^\rho(D)\big)$ uniformly over  $\{\xi^\varepsilon\}_{\varepsilon>0} \subset L^\rho(D)$, $\{\varphi^\varepsilon\} _{\varepsilon>0}\subset {\mathcal A}^N_2$, and $\varepsilon\in(0,1]$.  Note that 
 
  \begin{align}
\label{E:controlled_pross}
\nonumber v^{{\gamma(\varepsilon)}, \varphi^\varepsilon}_{\xi^\varepsilon}(t,x)= &\int_D G_t(x,y)\xi^\varepsilon(y)dy\\& + \sqrt{\gamma({\varepsilon})} \int_0^t\int_D G_{t-s}(x,y) \sigma_j(s, v^{{\gamma(\varepsilon)}, \varphi^\varepsilon}_{\xi^\varepsilon}(s))(y)dydB^j(s)\nonumber \\& -\int_0^t\int_D\partial_{y_i}G_{t-s}(x,y)g_i(s, v^{{\gamma(\varepsilon)}, \varphi^\varepsilon}_{\xi^\varepsilon}(s))(y)dyds\nonumber\\&+\int_0^t\int_D G_{t-s}(x,y)f(s, v^{{\gamma(\varepsilon)}, \varphi^\varepsilon}_{\xi^\varepsilon}(s))(y)dyds\nonumber \\& + \gamma(\varepsilon)\int_0^t\int_D G_{t-s}(x,y) \sigma_j(s, v^{{\gamma(\varepsilon)}, \varphi^\varepsilon}_{\xi^\varepsilon}(s))(y){\varphi_j^\varepsilon}(s)dyds\nonumber \\&:= Z_1^{\varepsilon}+Z_2^{\varepsilon}+Z_3^{\varepsilon}+Z_4^{\varepsilon}+Z_5^{\epsilon}
\end{align}

\noindent We show tightness of  $Z_{\ell}^{\varepsilon}$  for $\ell = 1, 2, 3, 4, 5$ in $C\big([0,T]: L^{\rho}(D)\big)$, and therefore assert the claim.  Since $\{\xi^{\varepsilon}\}_{\varepsilon>0} \subset L^{\rho}(D)$, the tightness of $Z_1^{\varepsilon}$ follows by the following lemma.





\begin{lemma} [{\cite[Lemma A.2]{C}}] \label{continuity}  Let $\xi\in L^\rho(D)$.  Then $(t\rightarrow G_t\xi)$ belongs to $C\big([0,T]: L^\rho(D)\big)$, and 
\begin{align*}
\xi \rightarrow \{t\rightarrow  G_t\xi\},
\end{align*}

\noindent  is a continuous map in $\xi$.

\end{lemma}

 \noindent As for the tightness of $Z_2^{\varepsilon}$, we employ Lemma \ref{Gyongy-s}.  Note that by (A4) and a slight modification of Proposition 4.4 in \cite{GR}

\begin{align*}
\sup_{\{\xi^\varepsilon\} \subset L^\rho(D)} \sup_{\{\varphi^\varepsilon\} \subset {\mathcal A}^N_2}\sup_{{\varepsilon}\in(0,1]}&{\mathbb E}(\sup_{t\leqslant T}|\sigma(s,\cdot)|_\rho^\rho)\\& \leqslant  \sup_{\{\xi^\varepsilon\} \subset L^\rho(D)}\sup_{\{\varphi^\varepsilon\} \subset {\mathcal A}^N_2}\sup_{{\varepsilon}\in(0,1]}{\mathbb E}(\sup_{t\leqslant T}|  v^{{\gamma(\varepsilon)}, \varphi^\varepsilon}_{\xi^\varepsilon}(s, \cdot)      |_\rho^\rho )\\& <\infty ,
\end{align*}

\noindent for all $s\in [0,T].$\\

\noindent Therefore, the assumption of Lemma \ref{Gyongy-s} is satisfied, and so the tightness of $Z_2^{\varepsilon}$ in $C\big([0,T]: L^\rho(D)\big)$ is concluded.\\

\noindent  As for the tightness of $Z_3^{\varepsilon}$, we mainly use Lemma \ref{Gyongy}.  Note that $g_i(t,x,r) := g_{i1}(x,t,r)+g_{i2}(t,r)$.  Therefore
 
 \begin{align*}
Z_3^{\varepsilon} &= \int_0^t\int_D\partial_{y_i}yG_{t-s}(x,y)g_i(s, v^{{\gamma(\varepsilon)}, \varphi^\varepsilon}_{\xi^\varepsilon}(s))(y)dyds \\& = \int_0^t\int_D\partial_{y_i}G_{t-s}(x,y)g_{i1}(s, v^{{\gamma(\varepsilon)}, \varphi^\varepsilon}_{\xi^\varepsilon}(s))(y)dyds\\&+\int_0^t\int_D\partial_{y_i}G_{t-s}(x,y)g_{i2}(s, v^{{\gamma(\varepsilon)}, \varphi^\varepsilon}_{\xi^\varepsilon}(s))(y)dyds\\&:= Z_{3,1}^\varepsilon + Z_{3,2}^\varepsilon.
 \end{align*}
 Note that $g_{i1}$ satisfies the linear growth condition:

$$\sup_{t\in[0,T]}\sup_{x\in[0,1]}|g_{i1}(t,x,r)|\leqslant K(1+|r|).$$\\
\noindent In Lemma \ref{Gyongy}, let $\zeta^\varepsilon(t,y) := g_{i1}(t,y, v^{\gamma(\varepsilon), \varphi^\varepsilon}_{\xi^\varepsilon}(t,y)).$ We have

$$\sup_{t\in [0,T]}| g_{i1}(s, v^{\gamma(\varepsilon), \varphi^\varepsilon}_{\xi^\varepsilon}(t, \cdot))|_1\leqslant K +K\sup_{t\in[0,T]} |v^{\varepsilon, \varphi^\varepsilon}_{\xi^\varepsilon}(t, \cdot)|_\rho.$$

\noindent  Let $\vartheta^\varepsilon := K +K\sup_{t\in[0,T]} |v^{\gamma(\varepsilon), \varphi^\varepsilon}_{\xi^\varepsilon}(t, \cdot)|_\rho$.  We have

\begin{align*}
\lim_{C\rightarrow\infty} &\sup_{\{\xi^\varepsilon\}\subset L^\rho(D)}\sup_{ \{\varphi^{\varepsilon}\} \subset {\mathcal A}_2^N}\sup_{\varepsilon \in (0,1]} P( K +K\sup_{t\in[0,T]} |v^{\gamma(\varepsilon), \varphi^\varepsilon}_{\xi^\varepsilon}(t, \cdot)|_\rho\geqslant C)\\& \leqslant \lim_{C\rightarrow\infty}\sup_{\{\xi^\varepsilon\} \subset L^\rho(D)}\sup_{\{\varphi ^\varepsilon\} \subset {\mathcal A}^N_2}\sup_{\varepsilon \in (0,1]} P(\sup_{t\in[0,T]} |v^{\gamma(\varepsilon), \varphi^\varepsilon}_{\xi^\varepsilon}(t, \cdot)|_\rho\geqslant \frac{C}{2}) \\&
+ \lim_{C\rightarrow\infty}\sup_{\{\xi^\varepsilon\} \subset L^\rho(D)}\sup_{\{\varphi^\varepsilon\} \subset {\mathcal A}^N_2}\sup_{\varepsilon  \in (0,1]}P(K\geqslant\frac{C}{2}). 
\end{align*}

\noindent Clearly the first term on the RHS of the immediate above display is equal to zero.   As for the second term,  it suffices to show that  

$$
\sup_{t\in [0,T]} |v^{\gamma(\varepsilon), \varphi^{\varepsilon}}_{\xi^{\varepsilon}}(t,.)|_{\rho},
$$
 \noindent is bounded in probability uniformly over $\varepsilon \in (0,1]$,  $\{\xi^\varepsilon\}_{\varepsilon>0} \subset L^\rho(D)$ and $\{\varphi^\varepsilon\}_{\varepsilon>0} \subset {\mathcal A}^N_2 $  i.e.,
 
\begin{equation}\label{bounded_probability}
 \lim_{C\rightarrow \infty}\sup_{\{\xi^\varepsilon\} \subset L^\rho(D)}\sup_{\{\varphi^\varepsilon\} \subset {\mathcal A}^N_2}\sup_{{\varepsilon}\in(0,1]}P\big( \sup_{t \leqslant T} |v^{\gamma({\varepsilon}), \varphi^{\varepsilon}}_{\xi^{\varepsilon}}(t,.)|_{\rho} \geqslant C\big) = 0.
\end{equation}
 
\noindent The proof of (\ref{bounded_probability}) is a direct consequence of a slight modification of Proposition 4.4 in \cite{GR} and Chebychev's inequality; however, a different proof is presented for the interested reader. To that end, the analysis of \cite{LS}  and \cite{FS} is used, and details are given for the convenience of the reader: \noindent In \cite{G}, Gy\"ongy and Rovira (2000) prove the existence and uniqueness of solutions to (1.1), by an approximation procedure.  They let $f_n(t,x,r)$, and $ g_{in}(t,x,r)$ and $\sigma_{jn}$ be Borel functions for every integer $n$, such that they are globally Lipschitz in $r\in\mathbb R$, and $f_n := f$, $g_{in} :=  g_i$ and $\sigma_{jn}:=\sigma_j$ for $|r|\leqslant n$, $f_n = g_{in}=\sigma_{jn} := 0$ for $|r| \geqslant n+1$.  Moreover, $f_n$, and $g_{in} := g_{i1n}+g_{i2n}$ and $\sigma_{jn}$ satisfy the same growth conditions as $f$, $g_i$, and $\sigma_j$, respectively. By (\cite [Proposition 4.1] {G}), there exists a unique solution, say $y^{\varepsilon}_n$ to the semilinear equation (\ref{E:1}) with $f$ , $g_i$ and $\sigma_j$ replaced by $f_n$, $g_{in}$ and $\sigma_{jn}$.  That is,  $y^{\varepsilon}_n$, is the unique solution to the truncated equation.  Furthermore, $y^{\varepsilon}_n$ converges to $y^\varepsilon$ in $C\big([0,T]: L^\rho(D)\big)$ in {{probability}}, uniformly over $\varepsilon\in (0,1]$ and $n \in {\mathbb N}$, as $n$ approaches infinity.  Now recall the class of controlled equations

\begin{align}\label{approx_control}
\frac{\partial}{\partial t}v^{\gamma(\varepsilon), \varphi^{\varepsilon}}_{\xi^\varepsilon} (t, x)=&\bigg[\frac{\partial}{\partial x_i} \bigg(b_{ij}(x)\frac{\partial}{\partial x_j} v^{\gamma(\varepsilon), \varphi^{\varepsilon}}_{\xi^\varepsilon} (t,x) + g_i(t,x, v^{\gamma(\varepsilon), \varphi^{\varepsilon}}_{\xi^\varepsilon}(t, x) )\bigg)\nonumber\\&+f(t,x, v^{\gamma(\varepsilon), \varphi^{\varepsilon}}_{\xi^\varepsilon} (t,x))\bigg]+\varepsilon \sigma_j(t, v^{\gamma(\varepsilon), \varphi^{\varepsilon}}_{\xi^\varepsilon} (t,x)){\varphi_j^\varepsilon}(t)\nonumber\\&+\sqrt{\varepsilon}{\sigma}_j(t, x, v^{\gamma(\varepsilon), \varphi^{\varepsilon}}_{\xi^\varepsilon} (t,x))\frac{d}{dt}B^j(t) 
 \end{align}

\noindent for all $t\geqslant0$ and $x\in D$. By the same analogy as Proposition 4.1 in \cite{G}, a unique solution to the approximated version of (\ref{approx_control}), say $v^{\gamma(\varepsilon), \varphi^{\varepsilon}}_{\xi^{\varepsilon},n}$, exists. Furthermore, $v^{\varepsilon, \varphi^\varepsilon}_{\xi^\varepsilon,n}$ converges to $v^{\varepsilon, \varphi^\varepsilon}_{\xi}$ in $C\big([0,T]: L^\rho(D)\big)$ in {{probability}}, uniformly over $\varepsilon\in (0,1]$, $\{\xi^\varepsilon\}_{\varepsilon>0} \subset L^\rho(D)$, $\{\varphi^\varepsilon\}_{\varepsilon>0}\subset {\mathcal P}^N_2$, and $n\in\mathbb N $. A slight modification of Proposition 4.4 in \cite{G} shows that

\begin{equation}\label{bounded_probability_approximation}
 \lim_{C\rightarrow \infty} \sup_{\varepsilon \in(0,1]} \sup_{\varphi^\varepsilon\subset {\mathcal P}^N_2} \sup_{x\subset L^\rho(D)} \sup_{n\in \mathbb N} P\big( \sup_{t\leqslant T} |v^{\gamma({\varepsilon}), \varphi^{\varepsilon}}_{{\xi^{\varepsilon},n}}(t,.)|_{\rho} \geqslant C\big) = 0.
\end{equation}
 
\noindent Observe that

\begin{align*}\label{ultimate}
\sup_{\{\xi^\varepsilon\} \subset L^\rho(D)}&\sup_{\{\varphi^\varepsilon\} \subset {\mathcal A}^N_2}\sup_{{\varepsilon}\in(0,1]} P \big(\sup_{t\leqslant T}|v^{\gamma(\varepsilon), \varphi^{\varepsilon}}_{\xi^{\varepsilon}}|_{\rho}\geqslant C\big)\leqslant \sup_{\{\xi^\varepsilon\} \subset L^\rho(D)}\sup_{\{\varphi^\varepsilon\} \subset {\mathcal A}^N_2}\\&\sup_{{\varepsilon}\in(0,1]} \sup_{n\in \mathbb N}P\bigg(\sup_{t\leqslant T}|v^{\gamma(\varepsilon), \varphi^{\varepsilon}}_{\xi^{\varepsilon}} -v^{\gamma({\varepsilon}), \varphi^{\varepsilon}}_{{\xi^{\varepsilon},n}}|_{\rho}\nonumber +\sup_{t\leq T}|v^{\gamma({\varepsilon}), \varphi^{\varepsilon}}_{{\xi^{\varepsilon},n}}   |_{\rho} \geqslant C\bigg)\\&\leqslant \sup_{\{\xi^\varepsilon\} \subset L^\rho(D)}\sup_{\{\varphi^\varepsilon\} \subset {\mathcal A}^N_2}\sup_{{\varepsilon}\in(0,1]} \sup_{n \in \mathbb N}P\big(\sup_{t\leqslant T}|v^{\gamma(\varepsilon), \varphi^{\varepsilon}}_{\xi^{\varepsilon}} -v^{\gamma({\varepsilon}), \varphi^{\varepsilon}}_{{\xi^{\varepsilon},n}}|_{\rho} \geqslant \frac{C}{2}\big)\\&+\sup_{\{\xi^\varepsilon\} \subset L^\rho(D)}\sup_{\{\varphi^\varepsilon\}  \subset {\mathcal A}^N_2}\sup_{{\varepsilon}\in(0,1]} \sup _{n \in \mathbb N} P\big(\sup_{t\leqslant T}|v^{\gamma({\varepsilon}), \varphi^{\varepsilon}}_{{\xi^{\varepsilon},n}}   |_{\rho}\geqslant \frac{C}{2}\big).
\end{align*}

\noindent By letting $C$ approach infinity, and exploiting the boundedness in probability of  $|v^{\gamma({\varepsilon}), \varphi^{\varepsilon}}_{\xi^{\varepsilon},n}|_{\rho}$, we get

\begin{align*}
&\lim_{C\rightarrow\infty}\sup_{\{\xi^\varepsilon\} \subset L^\rho(D)}\sup_{\{\varphi^\varepsilon\} \subset {\mathcal A}^N_2}\sup_{{\varepsilon}\in(0,1]} P \big(\sup_{t\leqslant T}|v^{\gamma(\varepsilon), \varphi^{\varepsilon}}_{\xi^{\varepsilon}}|_{\rho} \geqslant C\big)
 \\&\leqslant \lim_{C\rightarrow \infty} \sup_{\{\xi^\varepsilon\} \subset L^\rho(D)}\sup_{\{\varphi^\varepsilon\} \subset {\mathcal A}^N_2}\sup_{{\varepsilon}\in(0,1]} \sup_{n\in \mathbb N}P\big(\sup_{t\leqslant T}|v^{\gamma(\varepsilon), \varphi^{\varepsilon}}_{\xi^{\varepsilon}} -v^{\gamma({\varepsilon}), \varphi^{\varepsilon}}_{{\xi^{\varepsilon},n}}|_{\rho} \geqslant \frac{C}{2}\big). 
\end{align*}

\noindent Now by letting $n$ tend to infinity, due the convergence in probability of $v^{\gamma(\varepsilon), \varphi^{\varepsilon}}_{\xi^{\varepsilon},n}$  to $v^{\gamma(\varepsilon), \varphi^{\varepsilon}}_{\xi^{\varepsilon}}$, we conclude that

\begin{equation}\label{tightness}
 \lim_{C\rightarrow \infty} \sup_{\{\xi^\varepsilon\} \subset L^\rho(D)}\sup_{\{\varphi^\varepsilon\} \subset {\mathcal A}^N_2}\sup_{{\varepsilon}\in(0,1]}P\big( \sup_{t\leqslant T} |v^{\gamma({\varepsilon}), \varphi^{\varepsilon}}_{\xi^{\varepsilon}}(t,.)|_{\rho} \geqslant C\big) = 0.
\end{equation}
\noindent Therefore

$$\lim_{C\rightarrow\infty} \sup_{\{\xi^\varepsilon\} \subset L^\rho(D)}\sup_{\{\varphi^\varepsilon\} \subset {\mathcal A}^N_2}\sup_{{\varepsilon}\in(0,1]}P(\vartheta ^{\varepsilon}\geqslant C) = 0,$$

\noindent and the assumption of Lemma  \ref{Gyongy} is satisfied.  This establishes the tightness of $Z_3^{\epsilon}$.  The proof of tightness for $Z_4^{\epsilon}$ follows by the same analogy as $Z_3^{\epsilon}$, and thus omitted.   \\

\noindent  As for the tightness of $Z_5^{\varepsilon}$, we show that the H\"older norm is uniformly bounded in probability over $\varepsilon \in (0,1]$, $\{\xi^\varepsilon\}_{\varepsilon>0}\subset L^\rho(D)$ and
 $\{\varphi^\varepsilon\}_{\varepsilon>0} \subset {\mathcal A}^N_2.$   By Lemma \ref{lem:GR-3.1}(iv),  with $\gamma=2$, $q\geqslant 1$, $\displaystyle p= \frac{q}{q-1}$, $\displaystyle\kappa_p=1 - \frac{d}{2q}$ and any $\displaystyle \beta\in(0,\kappa_p-1/2)$, for $\xi, \xi+\zeta \in D$, and $t \in [0,T]$,
\begin{align*}
    &|Z^{\varepsilon}_5(t,x) - Z^{\varepsilon}_5(t,x+\zeta)| \leq C|\zeta|^\beta \left( \int_0^T \big|\sigma(s, v^{{\gamma(\varepsilon)}, \varphi^\varepsilon}_{\xi^\varepsilon}(s)(\cdot)\big|_{q}^2|\varphi^\varepsilon(s)|^2ds\right)^{\frac{1}{2}}\\
    &\leq C|\zeta|^\beta \sup_{s \in [0,T]} | v^{{\gamma(\varepsilon)}, \varphi^\varepsilon}_{\xi^\varepsilon}(s, \cdot)|_{q},
  \end{align*}
 \noindent where the boundedness of controls in $L^2([0, T] : {\mathbb R}^k)$ has been used.  Similarly, the H\"older continuity in time follows from Lemma \ref{lem:GR-3.1}(ii). For $\alpha \in (0, {\kappa}_p -1/2)$
  \begin{align*}
    &|Z^{\varepsilon}_5(t) - Z^{\varepsilon}_5(s)|_{L^\infty}\\
    & \leq C|t-s|^\alpha \left( \int_0^T \big|\sigma(s, v^{{\gamma(\varepsilon)}, \varphi^\varepsilon}_{\xi^\varepsilon}(s))(\cdot)\big|_{q}^2|\varphi^\varepsilon(s)|^2ds\right)^{\frac{1}{2}}\\
    &\leq C|t-s|^\alpha \sup_{s \in [0,T]} | v^{{\gamma(\varepsilon)}, \varphi^\varepsilon}_{\xi^\varepsilon}(s, \cdot)|_{q}.
     \end{align*}
 \noindent By choosing $q$ arbitrarily large, we can see that $Z^{\varepsilon}_3$ \textcolor{black}{is H\"older continuous in time and space and that the H\"older norm is uniformly bounded in probability over $\varepsilon \in (0,1]$, $\{\varphi^\varepsilon\}_{\varepsilon>0} \subset \mathcal{A}_2^N$, and $\{\xi^\varepsilon\}_{\varepsilon>0} \subset L^\rho(D)$. Therefore $Z^\varepsilon_5$ is uniformly tight in $C([0,T] : L^\rho(D))$.}\\

\noindent {\bf Step 2: Convergence}\\

Having the tightness of  $Z_{\ell}^{\varepsilon}$ for $\ell = 1, 2, 3, 4, 5$ at hand, by Prohorov's theorem, we can extract a subsequence along which each of the aforementioned processes converge in distribution to $Z_{\ell}^0$.  We aim to show that the respective limits are as follows:

\noindent \begin{align}
&\nonumber Z_1^0  := \int_0^1G_t(s,y)\xi(y) dy,\\&\nonumber Z_2^0 = 0,
\\&\nonumber
Z_3^0 := -\int_0^t\int_0^1 \partial_{y_i}G_{t-s}(x,y) g_i(s,\tilde v(s))(y)  dy ds,
\\&\nonumber 
Z_4^0 := \int_0^t\int_0^1 G_{t-s}(x,y) f(s,\tilde v(s))(y)  dy ds,
\\&\nonumber 
Z_5^0 := \int_0^t\int_0^1 G_{t-s}(x,y) \sigma_j(s,\tilde v)(y) \varphi_j(s) dy ds.
\end{align}

\noindent where  $\tilde v(t,x)$ is the limit in distribution of $v^{\gamma({\varepsilon}), \varphi^{\varepsilon}}_{\xi^{\varepsilon}}$ in $C\big([0,T]: L^\rho(D)\big)$.

\noindent The case $\ell=1$ follows from lemma (\ref{continuity}).  For $\ell=2$, let   

$${\mathcal J}^\varepsilon := \int_0^t\int_D G_{t-s}(x,y) \sigma_j(s, v^{{\gamma(\varepsilon)}, \varphi^\varepsilon}_{\xi^\varepsilon}(s))(y)dB^j(s).$$ 

\noindent Note that ${\mathcal J }^\varepsilon$ is tight by Lemma \ref{Gyongy-s}.  As a result, convergence of $Z_2^\varepsilon$ to zero as $\varepsilon\to 0$ follows readily.

As for $\ell=3$, we invoke the Skorokhod  Representation Theorem \cite{DK} and assume almost sure convergence on a larger, common probability space for the purpose of identification of limits.  Denote the RHS of $Z_3^0$ by ${\bar Z}_3^0$. We have

\begin{align}\label{g_convergence}
|Z_3^{\epsilon} - {\bar Z}_3^0| \nonumber & \leqslant \int_0^t\int_D |\partial_y G_{t-s}|\big(|1+|v^{\gamma(\varepsilon), \varphi^\varepsilon}_{{\xi}^\varepsilon}|^{\nu-1}+|\tilde v|^{\nu-1}\big) |v^{\gamma(\varepsilon), \varphi^\varepsilon}_{{\xi}^\varepsilon}-\tilde v|dy ds \\&\nonumber\leqslant \bigg(\sup_{x,t}|v^{\gamma(\varepsilon), \varphi^\varepsilon}_{{\xi}^\varepsilon}-\tilde v| \bigg) \Bigg[\sqrt{T} \bigg(\sup_t  |v^{\gamma(\varepsilon), \varphi^\varepsilon}_{{\xi}^\varepsilon}|_\rho^{\nu-1} +\sup_t|\tilde v|_\rho^{\nu-1}\bigg) \\& \times \bigg(\int_0^t\int_D |\partial_y G_{t-s}|^{\frac{\rho}{\rho-\nu+1}} dy ds \bigg)^{1-\nu/\rho +1/\rho}+\int_0^t\int_D |\partial_y G_{t-s}|dy ds\Bigg]  ,
\end{align}

\noindent where the Lipschitz property of $g_i$  with linearly growing constant (A5), and H\"older's inequality have been used.  Note that by estimate (E4) the integrals on the RHS of (\ref{g_convergence}) are finite.  Therefore, since $v^{\gamma(\varepsilon), \varphi^\varepsilon}_{{\xi}^\varepsilon}$ converges to $\tilde v$ as $\varepsilon \rightarrow 0$, the RHS of (\ref{g_convergence}) also converges to zero.  By the uniqueness of limits $Z_3^0 = {\bar Z}_3^0$.

 For $\ell=4$, we invoke the Skorokhod  Representation Theorem \cite{DK} again. Denote the RHS of  $Z_4^0$ by ${\bar Z}_4^0$.
 
\begin{align}\label{f_convergence}
 |Z_4^{\epsilon} - {\bar Z}_4^0| \nonumber& \leqslant \int_0^t\int_D |G_{t-s}| |v^{{\gamma(\varepsilon), \varphi^\varepsilon}}_{\xi^\varepsilon}-\tilde v| dy ds \\& \leqslant \sup_{x,t} |v^{{\gamma(\varepsilon), \varphi^\varepsilon}}_{\xi^\varepsilon}- \tilde v|  \int_0^t\int_D |G_{t-s}(x,y)| dy ds.
\end{align}

 \noindent Note that the RHS of (\ref{f_convergence}) converges to zero as $\varepsilon \rightarrow 0$ since $v^{\gamma(\varepsilon), \varphi^\varepsilon}_{{\xi}^\varepsilon}\to \tilde v$, and 

$$
\int_0^t\int_D |G_{t-s}(x,y)|dyds \leqslant C,
$$

 \noindent by estimate (E4).  Again by the uniqueness of limits $Z_4^0 = {\bar Z}_4^0$.  For $\ell=5$, we invoke the Skorokhod  Representation Theorem \cite{DK} again.  We have
\begin{align}\label{Eq:control_convergence}
|Z_5^{\epsilon} - \bar Z_5^0| &\nonumber\leqslant \int_0^t\int_D |G_{t-s}| |\sigma_j(s,v^{\gamma({\varepsilon}), \varphi^\varepsilon}_{{\xi}^\varepsilon}(s))(y)-\sigma_j(s,\tilde v (s))(y)| |\varphi^\varepsilon_j(s)|dyds\\& +\int_0^t\int_D |G_{t-s}| \sigma_j(s, \tilde v(s))(y) |\varphi^\varepsilon_j(s) - \varphi_j(s)|dsdy.
\end{align} 
\noindent The first term on the RHS of (\ref{Eq:control_convergence}) can be bounded above by

\begin{align} \label{cauchy-schwartz}
&\nonumber M\bigg( \int_0^t\int_0^1 |G_{t-s}|^2 |\sigma_j(s,v^{\gamma({\varepsilon}), \varphi^\varepsilon}_{\xi^\varepsilon}(s))(y)-\sigma_j(s,\tilde v(s))(y)|^2 dy ds \bigg)^{1/2}\\& \leqslant C (\sup_{x,t}| v^{\gamma({\varepsilon}), \phi^\varepsilon}_{{\xi}^\varepsilon} -\tilde v|),
\end{align} 
\noindent where the Cauchy-Schwartz inequality, properties of controls, estimate (E4), and (A5) have been used.  The first term on the RHS of (\ref{Eq:control_convergence}) thus converges to zero, since $v^{\gamma({\epsilon}), \varphi^\varepsilon}_{{\xi}^\varepsilon}\to \tilde v$ as $\varepsilon \to 0$ .  The second term on the RHS of (\ref{Eq:control_convergence}) also converges to zero as $\varepsilon \rightarrow 0$, since ${\varphi^{{\varepsilon}}_j\to \varphi_j}$, and 

\begin{align}
& \nonumber \int_0^t\int_D \big(G_{t-s} (x,y) \sigma_j(s,\tilde v(s))(y)\big)^2 dy ds  \\&\nonumber \leqslant \bigg(\int_0^t\int_D \big(G_{t-s}(x,y)\big)^{\frac{2\rho}{\rho-2}} dyds\bigg)^{1-2/\rho}  \\& \times\bigg(\int_0^t\int_D \big(\sigma_j(s,\tilde v(s))(y)\big)^\rho dy ds\bigg)^{2/\rho} \leqslant C,
\end{align}

\noindent   where (A4), and estimate (E4) have been used.  By the uniqueness of limits $Z_5^0 = {\bar Z}_5^0$.  This shows that along a subsequence, the controlled process converges to $\tilde v$ where $\tilde v$ satisfies the skeleton equation.  By unique solvability of the skeleton,  $\tilde v = v^{0,\varphi}_{\xi}(s)$.  This shows that the family of controlled processes, $v^{{\gamma(\varepsilon)}, \varphi^\varepsilon}_{\xi^\varepsilon}$, converges to the skeleton,   $v^{0,\varphi}_{\xi}(s)$.  The proof of Proposition 5.1 is thus complete.
 \end{proof}
 
 \subsection{Verification of Assumption (S1)}  Assumption (S1) follows by Theorem \ref{uniqueness}, and applying Theorem \ref{convergence} with $\gamma=0$.  
  
 \subsection{Verification of Assumption (S2)}  Assumption (S2) follows by applying Theorem  \ref{convergence}  with
$\gamma(r)=r$, $r\in[0,1)$.\\

\noindent {\it Proof of Theorem \ref{LDP}}.  Theorem \ref{LDP} follows immediately from verification of assumptions (S1) and (S2).\\

\noindent{\bf{Acknowledgement.}}  
The author would like to thank Professor Hui Wang of Brown University for fruitful discussions on the manuscript.

\bibliographystyle{amsplain}

\end{document}